\documentclass[11pt]{amsart}

\usepackage[T1]{fontenc}
\usepackage[utf8]{inputenc}
\usepackage{lmodern,microtype}
\usepackage{geometry}
\usepackage{amsmath,amssymb,amsthm,mathtools}
\usepackage{enumitem}
\usepackage{aliascnt}
\usepackage{xcolor,booktabs,tabularx}

\usepackage[hidelinks]{hyperref}
\usepackage[nameinlink,noabbrev]{cleveref}

\numberwithin{equation}{section}

\newtheorem{theorem}{Theorem}[section]
\newaliascnt{proposition}{theorem}
\newtheorem{proposition}[proposition]{Proposition}
\aliascntresetthe{proposition}
\newaliascnt{lemma}{theorem}
\newtheorem{lemma}[lemma]{Lemma}
\aliascntresetthe{lemma}
\newaliascnt{corollary}{theorem}

\aliascntresetthe{corollary}
\newaliascnt{question}{theorem}

\aliascntresetthe{question}

\theoremstyle{definition}
\newaliascnt{definition}{theorem}

\aliascntresetthe{definition}

\theoremstyle{remark}
\newaliascnt{remark}{theorem}

\aliascntresetthe{remark}

\crefname{theorem}{Theorem}{Theorems}
\Crefname{theorem}{Theorem}{Theorems}
\crefname{proposition}{Proposition}{Propositions}
\Crefname{proposition}{Proposition}{Propositions}
\crefname{lemma}{Lemma}{Lemmas}
\Crefname{lemma}{Lemma}{Lemmas}
\crefname{corollary}{Corollary}{Corollaries}
\Crefname{corollary}{Corollary}{Corollaries}
\crefname{question}{Question}{Questions}
\Crefname{question}{Question}{Questions}
\crefname{definition}{Definition}{Definitions}
\Crefname{definition}{Definition}{Definitions}
\crefname{remark}{Remark}{Remarks}
\Crefname{remark}{Remark}{Remarks}

\newcommand{\C}{\mathbb C}
\newcommand{\D}{\mathbb D}

\newcommand{\Aut}{\operatorname{Aut}}
\newcommand{\Sc}{\mathcal S_c(2)}
\newcommand{\Ec}{\mathcal E_c}
\newcommand{\Af}{\mathcal A}

\author{Luka Boc Thaler}
\thanks{The author was supported by the research program P1-0291 and project J1-70033, all from ARRS, Republic of Slovenia}

\address{L. Boc Thaler: Faculty of Mathematics and Physics, University of Ljubljana, SI--1000 Ljubljana, Slovenia. Institute of Mathematics, Physics and Mechanics, Jadranska 19, 1000 Ljubljana, Slovenia.} \email{luka.boc@fmf.uni-lj.si}

\title{On the rank-one limit sets for maps in the constant-Jacobian shear group of $\C^2$}
\date{}
\begin{document}
\begin{abstract}

Let $\Sc$ be the subgroup of $\Aut(\C^2)$ generated by affine maps and shear automorphisms with constant Jacobian. We determine the possible conformal types of the rank-one limit set of an invariant recurrent Fatou component when the induced dynamics has infinite order. 
In particular we show that, under this infinite-order hypothesis, the limit set of a finite composition of H\'enon maps (polynomial or transcendental), can only be conformally equivalent to a disk or an annulus. Thus we exclude the remaining three conformal types, i.e. $\D^*$, $\C^*$ and $\C$, from the classification result by Arosio, Benini, Forn\ae ss and Peters \cite{ABFP2019}. We also prove that the set of fixed points of a finite composition of H\'enon maps (polynomial or transcendental) contains no entire curves.
\end{abstract}
\maketitle

\section{Introduction}
The classification of Fatou components plays a central role in the study of global holomorphic dynamics. For rational maps of degree at least two, the classical theory of Fatou and
Julia, together with the work of Siegel, Arnold and Herman, provides the
standard classification of periodic Fatou components
\cite{Fatou1919,Julia1918,Siegel1942,Arnold1961,Herman1979}: every such
component is an immediate attracting or parabolic basin, a Siegel disk,
or a Herman ring. Sullivan's no wandering domains theorem, which implies that every Fatou component of a rational map is eventually periodic \cite{Sullivan1985}, completes the classification.
For transcendental entire functions, the periodic components are attracting
or parabolic basins, Siegel disks, or Baker domains \cite{Bergweiler1993},
however wandering domains may also occur \cite{Baker1976,Baker1984}.
Their internal dynamics of wandering domains have been studied extensively, see \cite{BEFRS2022, BRS2013}.

In several complex variables the picture is considerably more intricate.
For holomorphic endomorphisms of $\mathbb{P}^2$ of degree at least two,
recurrent periodic Fatou components were classified by Forn{\ae}ss and
Sibony \cite{FS1995}, with Ueda excluding the punctured disk as a 
limit set \cite{Ueda2008}.
For non-recurrent invariant components, Lyubich and Peters \cite{LP2014} obtained a
classification under the additional hypothesis of a unique limit set. The author, together with Forn{\ae}ss and Peters, subsequently constructed an example of endomorphism  with an invariant non-recurrent Fatou component whose limit set is a punctured disk, thereby  realizing the remaining
possibility in that classification \cite{BTFP2015}.

Astorg, Buff, Dujardin, Peters and Raissy constructed polynomial
skew-products, which extend to endomorphisms of $\mathbb{P}^2$, with wandering
Fatou components \cite{ABDPR2016}, showing that Sullivan's theorem does
not extend to a several variable setting.
Further examples were obtained by the author in collaboration with Astorg and Peters \cite{ABTP2023, ABT2026}. 

For polynomial automorphisms of $\mathbb{C}^2$ of H\'enon type, the
 work of Bedford and Smillie \cite{BS1991a,BS1991b}, together
with results of Forn{\ae}ss and Sibony  \cite{FS1995},
describes the recurrent periodic components. In \cite{Krikorian2025},  R.~Krikorian  recently proved the existence of attracting Herman rings for dissipative  H\'enon maps, realizing the remaining possibility in that classification.

Under the moderate dissipativity assumption $|\det DF|<d^{-2}$, where $d$
is the degree of the H\'enon map, Lyubich and Peters proved that every
invariant Fatou component with bounded forward orbits is an attracting
basin, a semi-parabolic basin, or the basin of a rotation domain
\cite{LP2014}.
Surprisingly, Berger and Biebler showed that wandering components also occur for H\'enon maps \cite{BB2023}.

This finally brings us to the work of Arosio, Benini, Forn{\ae}ss and Peters \cite{ABFP2019}, who obtained the following classification, which motivated the present work.

\begin{theorem}\label{thm:abfp}
Let $F$ be a holomorphic automorphism of $\mathbb{C}^2$ with constant
Jacobian $\delta$ and let $\Omega$ be an invariant recurrent Fatou
component for $F$. Then there exists a holomorphic retraction $\rho$
from $\Omega$ to a closed complex submanifold $\Sigma \subset \Omega$,
called the \emph{limit set}, such that for all limit maps $h$
there exists an automorphism $\eta$ of $\Sigma$ such that $h = \eta \circ \rho$. Every orbit converges to $\Sigma$, and $F|_{\Sigma} \colon \Sigma \to \Sigma$
is an automorphism. Moreover:
\begin{enumerate}

\item If $\dim \Sigma = 0$, then $\Omega$ is the basin of an attracting
fixed point, and is biholomorphically equivalent to $\mathbb{C}^2$.

\item If $\dim \Sigma = 1$, either $\Sigma$ is biholomorphic to a circular
domain $A$, and there exists a biholomorphism from $\Omega$ to
$A \times \mathbb{C}$ which conjugates the map $F$ to
$$
(z,w) \longmapsto
\left(
e^{2\pi i\theta}z,
\frac{\delta}{e^{2\pi i\theta}}w
\right),
$$
where $\theta$ is irrational, or there exists $j \in \mathbb{N}$ such that
$F^j|_{\Sigma} = \operatorname{id}_{\Sigma}$,
and there exists a biholomorphism from $\Omega$ to
$\Sigma \times \mathbb{C}$ which conjugates the map $F^j$ to
$$
(z,w) \longmapsto (z,\delta^j w).
$$

\item $\dim \Sigma = 2$ if and only if $|\delta|=1$. In this case there
exists a sequence of iterates converging to the identity on $\Omega$.

\end{enumerate}
\end{theorem}

By a circular domain in statement $(2)$ we mean either the disk, the punctured disk, an annulus, the complex plane, or the punctured plane. It is known that the disk, an annulus, and the complex plane can be realized, while the realizability of the punctured disk and the punctured plane remained open. 
\medskip

We use $\text{Aut}_c(\C^2)$ to denote the group of automorphisms of $\C^2$ with constant Jacobian. Furthermore let $\Af$ be the affine automorphism group of $\C^2$, and let $\Ec$ consist of the maps
\begin{equation}\label{eq:elementary}
 E(z,w)=(az+g(w),bw+c),\qquad a,b\in\C^*,\quad c\in\C,\quad g\in\mathcal O(\C).
\end{equation}
Set $\Sc=\langle\Af,\Ec\rangle$. In the terminology of Ahern and Forstneri\v c~\cite[\S7]{AhernForstneric}, this is the constant-Jacobian shear group which is known to be dense in $\text{Aut}_c(\C^2)$, with respect to compact open  topology. C. de Fabritiis \cite{deFabritiis} proved  that $\Sc$ has an amalgamated structure
\begin{equation*}
 \Sc=\Af*_{\Af\cap\Ec}\Ec.
\end{equation*}
 This implies that every $F\in\Sc$ is  conjugate in $\Sc$ to an element of $\Af$ or $\Ec$, or to a finite composition of  H\'enon type maps
\begin{equation}\label{eq:henon}
 H_j(z,w)=(f_j(z)-\delta_jw,z),\qquad \delta_j\in\C^*,
\end{equation}
in which each $f_j$ is a non-affine entire function. The following is the main result of this paper.
\begin{theorem}\label{thm:main}
Let $F\in\Sc$, and let $\Omega$ be an invariant recurrent Fatou component whose limit set $\Sigma$ has complex dimension one. Suppose that $F|_\Sigma$ has infinite order, i.e. $F^j|_\Sigma\neq \text{id}$ for all $j$. Then $\Sigma$ is biholomorphic to $\C$, $\D$, or an annulus. More precisely:
\begin{enumerate}
\item If $F$ is conjugate in $\Sc$ to an element in $\Af\cup\Ec$, then $\Omega=\C^2$ and $\Sigma\simeq\C$.
\item If $F$ is conjugate in $\Sc$ to a finite composition of  H\'enon type maps \eqref{eq:henon}, then $\Sigma$ is biholomorphic to a disk or an annulus.
\end{enumerate}
\end{theorem}

The result shows that  $\D^*$ and $\C^*$ do not occur as infinite-order recurrent rank-one limit sets for elements of $\Sc$. Such classification result is known to hold for the group of polynomial automorphisms $\text{Aut}_{P}(\C^2)\subsetneq \Sc$, and follows from the work of Bedford and Smillie, Forn{\ae}ss and Sibony and finally Krikorian  \cite{BS1991b, FS1995, Krikorian2025}. The infinite-order condition on $F|_\Sigma$ is not needed for $(1)$. Moreover, this condition is trivially satisfied for  $F\in \text{Aut}_{P}(\C^2)$ in $(2)$, since its periodic points $\text{Fix}(F^k)$ are discrete for all $k>0$. However, it is not known whether the same property holds for all $F\in\Sc$ which are not conjugate to elements in $\Af\cup\Ec$, see  \cite[Proposition 3.2]{ABFP2019} and the preceding
discussion therein. Here we prove that the set of periodic points of period at most $k$ contains no entire holomorphic curves. In fact we can say much more:

\begin{theorem}\label{thm:fix} Let $F\in\mathcal S_c(2)$ be a finite composition of H\'enon-type maps. For every $k\ge1$, if $S$ is a connected Riemann surface and
$\varphi:S\longrightarrow\text{Fix}(F^k)$
is a non-constant proper holomorphic map, then $S$ is hyperbolic and has no puncture ends.
\end{theorem}

%\begin{theorem}\label{thm:fix}
%Let $F\in\Sc$ be a finite composition of H\'enon type maps \eqref{eq:henon}. Then there %exist no non-constant holomorphic maps $\gamma:\C\to \text{Fix}(F^k)$, for any $k\geq 1$.
%\end{theorem}

{\it Organization of the paper:} To motivate the proof of \Cref{thm:main}, we first illustrate in Section 2 the underlying idea in the polynomial case, where the argument can be carried out using the maximum modulus, while the general setting will require the use of Nevanlinna characteristic. In Section 3, we review some classical results from the Nevanlinna theory, which  are later used in Sections 4 and  5 to prove \cref{thm:main} and \cref{thm:fix} respectively.

\section{Exclusion of $\C$, $\C^*$ and $\D^*$ for polynomial H\'enon automorphisms}

For a finite composition of  H\'enon maps, a classical proof of this exclusion is quite elementary. However, the proof relies on the filtration property of these maps, which no longer holds in the general setting. We therefore first seek an alternative proof in the polynomial setting that can later be generalized to the non-polynomial setting.

Let $F=H_m\circ\cdots\circ H_1$, be a finite composition of polynomial H\'enon maps
$$
    H_j(z,w)=\bigl(P_j(z)-\delta_j w,z\bigr),
$$
where $\delta_j\in\C^*$ and each $P_j$ is a polynomial of degree $d_j\ge 2$. Let $\Omega$ be an invariant recurrent rank-one Fatou component of $F$, with limit set $\Sigma$, and suppose that
$F|_\Sigma$ has infinite order. Then by \cref{thm:abfp} there exists a circular domain $A$ and a biholomorphism $\gamma:A\to\Sigma$ such that
$$
    F\circ\gamma(t)=\gamma(\lambda t),
    \qquad |\lambda|=1,
$$
where $\lambda$ is not a root of unity. Write   $\gamma(t)=\bigl(x_1(t),x_{0}(t)\bigr)$,
and define holomorphic functions $x_2,\ldots,x_{m} \in\mathcal{O}(A)$ by
$$
    H_j\circ\cdots\circ H_1\circ\gamma(t)
       =\bigl(x_{j+1}(t),x_{j}(t)\bigr).
$$
This gives 
\begin{equation}\label{eq:poly-cyclic}
    P_j(x_{j}(t))
      =x_{j+1}(t)+\delta_jx_{j-1}(t),
    \qquad 1\le j\le m,
\end{equation}
together with the closing relations 
\begin{equation}\label{eq:poly-closing}
    x_{m+1}(t):=x_1(\lambda t) \quad\text{ and }\quad  x_{m}(t):=x_{0}(\lambda t).
\end{equation}
Thus we say that the functions $x_1,\ldots,x_{m}$ form a cycle. We first establish the maximum-modulus estimate that will be used throughout.

\begin{lemma}\label{lem:poly-max}
Let $U$ be equal to $\C$ or $\{t\in\C :|t|>R_0\}$ for some $R_0\ge 0$, and
suppose that $x_1,\ldots,x_{m}$ are holomorphic on $U$ and satisfy
the cyclic equations \eqref{eq:poly-cyclic} and \eqref{eq:poly-closing}.
Then there exists a constant $C_F>0$, such that
$$
   M(r):= \max_{1\le j\le m}\max_{|t|=r}|x_j(t)|     \le C_F
$$
for every circle $\{|t|=r\}\subset U$. In particular, functions $x_j$ extend holomorphically across $\infty$.
\end{lemma}

\begin{proof}
First observe that there exist constants $c>0$ and $R>0$, such that for each $1\leq j\leq m$ we have
\begin{equation}\label{eq:poly-lower}
    |P_j(\zeta)|
       \ge c|\zeta|^{d_j},
    \qquad |\zeta|\ge R.
\end{equation}
For each $r>0$, for which $S_r:=\{|t|=r\}$ is contained in $U$, we choose an index $k$ and a point $t_r\in S_r$, such that $|x_k(t_r)|=M(r)$.

The cyclic equation whose polynomial acts on $x_k$ has the form
\begin{equation}\label{eq:poly-max-eq}
    P_{k}(x_k(t))
      =
      x_{k+1}(t)
      +
      \delta_{k}x_{k-1}(t)
\end{equation}
where we use \eqref{eq:poly-closing} when $k=1$ or $k=m$. 

Suppose first that $M(r)\ge R$.
Evaluating \eqref{eq:poly-max-eq} at $t=t_r$ and using
\eqref{eq:poly-lower}, we obtain
$$
    cM(r)^{d_{k}}
       \le
       |P_{k}(x_k(t_r))|.
$$
On the other hand, since multiplication by $\lambda$ preserves the circle $S_r$,
$$
\begin{aligned}
    |P_{k}(x_k(t_r))|
    &\le
      \bigl(1+|\delta_{k}|\bigr)M(r).
\end{aligned}
$$
Hence
$$
    cM(r)^{d_{k}-1}
       \le 1+|\delta_{k}|.
$$
Since $d_{k}\ge2$,
$$
    M(r)
      \le
      \left(
        \frac{1+|\delta_{k}|}{c}
      \right)^{1/(d_{k}-1)}.
$$

There are only finitely many factors. Therefore, with
$$
    C_F:=
    \max_{1\le j\le m}
    \left\{
      R,\,
      \left(
        \frac{1+|\delta_j|}{c}
      \right)^{1/(d_j-1)}
    \right\},
$$
we obtain  $M(r)\le C_F$ for every admissible $r$. Thus, each $x_j$ is
bounded near $\infty$, so $\infty$ is a removable singularity for each $x_j$.
\end{proof}

 {\bf Exclusion of $\C$:} If $A= \C$, then \cref{lem:poly-max} implies that all entire functions $x_j$ are bounded and therefore by Liouville’s theorem also constant. Thus, $\gamma:A\to \Sigma$ is constant, which is a contradiction.  
 \medskip 
 
{\bf Exclusion of $\C^*$:} Let $A= \C^*$, let $x_1,\ldots,x_m$, as above, be holomorphic on $\C^*$. By applying \cref{lem:poly-max} to $X_j(t)=x_j(1/t)$ we can see that all maps $x_j$ extend to an entire function, and therefore, by the same argument as before, they have to be constant. Again, this forces $\gamma:A\to \Sigma$ to be constant, which is a contradiction. 
\medskip 

{\bf Exclusion of $\D^*$:} Now assume that $A=\D^*$ and apply \cref{lem:poly-max} to $X_j(t)=x_j(1/t)$. Thus, we can extend $\gamma:\D^*\to\C^2$ to a non-constant holomorphic function $\hat \gamma:\D\to\C^2$.
We claim that $p:=\hat \gamma(0)\in\Omega$. Fix $0<r<1$ and let $S_r=\{|t|=r\}$. The curve $\gamma(S_r)$ is compact and contained in $\Omega$, therefore there exists $\epsilon>0$ such that its $\epsilon$-tubular neighborhood $T_{\epsilon, r}:=\{\gamma(t)+v: |t|=r \text{ and } \|v\|\le \epsilon\}$ is also contained in $\Omega $. By local boundedness of the iterates on $\Omega$, there exists $M>0$
such that $ \|F^n|_{T_{\epsilon, r}}\|\le M$ for all $n\ge 0$. By the maximum principle
we also have $\|F^n\|\le M$  on the compact $\widehat{T}_{\epsilon, r}=\{\hat \gamma(t)+v: |t|\le r \text{ and } \|v\|\le \epsilon\}$ for all $n\ge 0$. 
 In particular, at $t=0$, $\|F^n(p+v)\|\le  M$,
for all  $v\in \mathbb{B}(0,\varepsilon)$ and all $n\ge 0$. Thus the family $\{F^n\}$ is locally bounded in a neighborhood of $p$, and hence is normal there. Therefore $p$ belongs to the Fatou
set. Since the Fatou set is open it follows that $p\in \hat\gamma(\D)\subset\Omega$.  But this leads to a contradiction since $\Sigma=\gamma(\D^*)$ is closed in $\Omega$. Hence, punctured disk cannot be the limit set.

\section{Nevanlinna characteristic}

In this section, we collect several results on the Nevanlinna characteristic for entire functions and for holomorphic functions on the annulus $\mathcal A_{r,R}:=\{r<|z|<R\}$, where $0\le r<R\leq\infty$.

\subsection{Entire functions}
For an entire function $f$ the Nevanlinna characteristic agrees with the proximity function, hence 
$$
 T(r,f)=m(r,f)=\frac1{2\pi}\int_0^{2\pi}\log^+|f(re^{i\theta})|\,d\theta,
 \qquad \log^+x:=\max\{\log x,0\}.
$$

The following result appears in Ahern and Rudin \cite[Theorem~1.3(iv),(v)]{AhernRudin}, who attribute it to earlier work of Clunie \cite{Clunie} and Hayman \cite{Hayman}.

\begin{lemma}\label{lem:clunie}
Let $f$ be a nonconstant entire function and let $g$ be entire.
\begin{enumerate}
\item If $g$ is a polynomial of degree $d$, then
\begin{equation*}
 \frac{T(r,g\circ f)}{T(r,f)}\longrightarrow d
 \qquad (r\to\infty).
\end{equation*}
\item If $g$ is transcendental entire, then
\begin{equation*}
 \frac{T(r,g\circ f)}{T(r,f)}\longrightarrow\infty
 \qquad (r\to\infty).
\end{equation*}
\end{enumerate}
\end{lemma}

We shall also use the elementary inequalities
\begin{align*}
 T(r,g+f)&\le T(r,g)+T(r,f)+O(1),\label{eq:sumTplane}\\
 T(r,f(\lambda\,\cdot))&=T(r,f),\qquad |\lambda|=1.%\label{eq:rotTplane}
\end{align*}
The first follows from $\log^+|u+v|\le\log^+|u|+\log^+|v|+\log2$, while the second uses the rotation invariance.

\subsection{Holomorphic functions on an annulus}
In his paper Korhonen \cite{Korhonen} collected earlier results on Nevanlinna theory outside a disk and further developed the theory for meromorphic functions on an open annulus. Here we will only consider holomorphic functions on the infinite open annulus $\mathcal A_{R_0,\infty}$, where $R_0\geq 0$.
As before, the annular characteristic of a holomorphic function $f\in \mathcal{O}(\mathcal A_{R_0,\infty})$ agrees with the annular proximity function, and is given by
$$
\begin{aligned}
 T_A(\rho,r,f)=m_A(\rho,r,f)
 =\frac1{2\pi}\int_0^{2\pi}\log^+|f(\rho e^{i\theta})|\,d\theta +\frac1{2\pi}\int_0^{2\pi}\log^+|f(re^{i\theta})|\,d\theta,
\end{aligned}
$$
where $R_0<\rho<r$, see \cite[\S2]{Korhonen}. Annular characteristic has the following elementary properties:
\begin{align*}
 T_A(\rho,r,f+g)&\leq T_A(\rho,r,f)+T_A(\rho,r,g)+O(1),\\ %\label{eq:sumTA}\\
 T_A(\rho,r,f(\lambda\,\cdot))&=T_A(\rho,r,f),\qquad|\lambda|=1, %\label{eq:rotTA}\\
\end{align*}

\begin{lemma}\label{lem:laurent}
Let $f$ be holomorphic on the annulus $\mathcal{A}_{R_0 ,\infty}$ with an essential singularity at infinity. Then for every fixed $\rho>R_0$ we have
\begin{equation}\label{eq:essentialgrowth}
 \frac{T_A(\rho,r,f)}{\log r}\longrightarrow\infty \qquad (\text{as }r\to\infty).
\end{equation}
\end{lemma}
\begin{proof}
Split the Laurent series of $f$ into its nonnegative and negative powers. The nonnegative part $g$ is entire, while the negative part $h$ is $O(1/z)$ at infinity. Hence
$$
 \frac1{2\pi}\int_0^{2\pi}\log^+|f(re^{i\theta})|\,d\theta
 =T(r,g)+O(1)
$$
where $T(r,g)$ is the usual Nevanlinna characteristic for entire functions. Since $\rho$ is fixed, the inner-circle proximity function is constant with respect to $r$, hence
\begin{equation*}\label{eq:laurentT}
 T_A(\rho,r,f)=T(r,g)+O(1)\qquad(\text{as } r\to\infty).
\end{equation*}
An essential singularity of $f$ at infinity is equivalent to $g$ being transcendental entire. In that case $T(r,g)/\log r\to\infty$,  see $(2)$ in \cref{lem:clunie}, which concludes the proof.
\end{proof}

The following lemma is an analogue of \cref{lem:clunie} in the annular case.

\begin{lemma}\label{lem:fourfiber}
Let $f$ be a holomorphic function on $\mathcal{A}_{R_0 ,\infty}$ with essential singularity at infinity and let $g$ be entire. Fix $\rho>R_0$.
\begin{enumerate}
\item If $g$ is a polynomial of degree $d$, then
\begin{equation*}
 \frac{T_A(\rho,r,g\circ f)}{T_A(\rho,r,f)}\longrightarrow d
 \qquad (r\to\infty).
\end{equation*}
\item If $g$ is transcendental entire, then
\begin{equation*}
 \frac{T_A(\rho,r,g\circ f)}{T_A(\rho,r,f)}\longrightarrow\infty
 \qquad (r\to\infty).
\end{equation*}
\end{enumerate}
\end{lemma}
\begin{proof}
Part $(1)$ simply follows from $\log^+|p(w)|=d\log^+|w|+O(1)$ and \cref{lem:laurent}.
For part $(2)$, let $n\geq 1$ be arbitrary. Since $g$ is
transcendental entire, the Great Picard Theorem gives a
finite value $b$ with $n$ distinct preimages
$a_1,\ldots,a_n$ under $g$.

Since $g(a_j)=b$, there exists a constant $C_n>0$ such
that
$$
 \sum_{j=1}^n\log^+\frac1{|w-a_j|}
 \leq \log^+\frac1{|g(w)-b|}+C_n
$$
for all $w\in\C$. Indeed, near each $a_j$ we have
$g(w)-b=O(w-a_j)$, while outside disjoint neighborhoods
of these points the sum on the left is bounded.
Substituting $w=f(z)$ and integrating over the two
circles gives
$$
 \sum_{j=1}^n
 m_A\!\left(\rho,r,\frac1{f-a_j}\right)
 \leq
 m_A\!\left(\rho,r,\frac1{g\circ f-b}\right)+O(1).
$$

Moreover, since the values $a_j$ are distinct and
satisfy $g(a_j)=b$, every zero of $f-a_j$ is also
a zero of $g\circ f-b$, with at least the same
multiplicity. Consequently, for annular counting function we have
$$
 \sum_{j=1}^n N_A(\rho,r,a_j;f)
 \leq N_A(\rho,r,b;g\circ f).
$$

By Korhonen's First Main Theorem
\cite[Theorem~4.1]{Korhonen},
$$
\begin{aligned}
 nT_A(\rho,r,f)
 &=\sum_{j=1}^n
 \left(
 m_A\!\left(\rho,r,\frac1{f-a_j}\right)
 +N_A(\rho,r,a_j;f)
 \right)+O(\log r)\\
 &\leq
 m_A\!\left(\rho,r,\frac1{g\circ f-b}\right)
 +N_A(\rho,r,b;g\circ f)+O(\log r)\\
 &=T_A(\rho,r,g\circ f)+O(\log r).
\end{aligned}
$$
Here the implied constants may depend on $n$ and
$\rho$, which are fixed. Dividing by $T_A(\rho,r,f)$ and using
\eqref{eq:essentialgrowth}, we obtain
$$
 \liminf_{r\to\infty}
 \frac{T_A(\rho,r,g\circ f)}{T_A(\rho,r,f)}
 \geq n.
$$
Since $n$ is arbitrary, the assertion follows.
\end{proof}

\section{Exclusion of three circular types for a general H\'enon composition}

Let $F=H_m\circ\cdots\circ H_1$, be a finite composition of H\'enon type maps \eqref{eq:henon}. In Section 2 we have already treated the case when all $H_j$ are polynomial maps, therefore from now on we always assume that at least one of them is transcendental. Let $\Omega$ be an invariant recurrent Fatou component with rank-one limit set $\Sigma$, and assume that $F|_\Sigma$ has infinite order.
By~\cref{thm:abfp}, there are a circular domain $A$ and a non-constant biholomorphism
$\gamma:A\to\Sigma$ such that
\begin{equation}\label{eq:gammaeq}
 F\circ\gamma(t)=\gamma(\lambda t),\qquad
 |\lambda|=1,\qquad \lambda^n\ne1\quad(n\ge1).
\end{equation}
As in Section~2, write $\gamma(t)=(x_1(t),x_0(t))$ and define the
subsequent coordinates by
$$
 H_j\circ\cdots\circ H_1\circ\gamma(t)
    =(x_{j+1}(t),x_j(t)),\qquad 1\le j\le m.
$$
Then
\begin{equation}\label{eq:cyclic}
 \begin{gathered}
 f_j(x_j(t))=x_{j+1}(t)+\delta_jx_{j-1}(t),
       \qquad 1\le j\le m,\\
 x_{m+1}(t):=x_1(\lambda t),\qquad
 x_m(t):=x_0(\lambda t),
 \end{gathered}
\end{equation}
where at least one of $f_j$ is transcendental entire function. These are the analogues of~\eqref{eq:poly-cyclic} and
\eqref{eq:poly-closing}. As before observe that $x_1,\ldots,x_m$ form a cycle. Moreover, by rotation invariance, the planar and annular characteristics of $x_0$ and $x_{m+1}$ agree with those
of $x_m$ and $x_1$, respectively.

\begin{proposition}\label{prop:planegrowth}
If $x_1,\ldots,x_m$ are entire on $\C$, then they are all constant.
\end{proposition}
\begin{proof}
Suppose some coordinate is nonconstant. Let
 $M(r):=\max_{1\le j\le m}T(r,x_j)$, and note that  $M(r)\to\infty$. Let 
 $$U(r)=\left(\frac{T(r,x_1)}{M(r)},\ldots,\frac{T(r,x_m)}{M(r)}\right)\in [0,1]^m$$
and choose a sequence $r_\nu\to\infty$ so that $U(r_\nu)$ converges to some vector 
$(u_1,\ldots,u_m)\in [0,1]^m$. Observe that $\max_j u_j=1$ and choose $k$ with
$u_k=1$. In this case $x_k$ is clearly not constant.

The cyclic identity~\eqref{eq:cyclic}, the sum inequality, and rotation
invariance from Section~3 give
\begin{equation}\label{eq:cyclicTplane}
 T(r,f_k\circ x_k)
 \le  T(r,x_{k-1})+T(r,x_{k+1})+O(1),
\end{equation}
where in characteristic estimates we read indices cyclically, so
$$T(r,x_{0})=T(r,x_{m})\quad \text{and}\quad T(r,x_{m+1})=T(r,x_{1}).$$
If $f_k$ is transcendental, part~(2) of~\cref{lem:clunie} implies
$T(r,f_k\circ x_k)/T(r,x_k)\to\infty$. Dividing
\eqref{eq:cyclicTplane} by $M(r_\nu)$ gives 
$$
 \frac{T(r_\nu,f_k\circ x_k)}{T(r_\nu,x_k)}(1+o(1))
 \le  u_{k-1}+u_{k+1}+o(1),\qquad (r_\nu\to\infty)
$$
which brings us to a contradiction.
\medskip

If $f_k$ is a polynomial of degree $d_k\ge2$, part~(1)
of~\cref{lem:clunie} gives
$$
 d_k u_k\le u_{k-1}+u_{k+1}\le2.
$$
Since $u_k=1$, necessarily $d_k=2$ and
$u_{k-1}=u_{k+1}=1$. Thus maximal normalized growth propagates to the
next index $k+1$. Continuing around the finite cycle, we eventually
reach an index $j$ for which $f_j$ is transcendental, giving the preceding contradiction.
\end{proof}

{\bf Exclusion of $\C$:} By \cref{prop:planegrowth} all the maps $x_j(t)$ are constant. In particular this implies that the map $\gamma:\C\to\Sigma$ is constant, which brings us to a contradiction.

\begin{proposition}\label{prop:exteriorgrowth}
If $x_1,\ldots,x_m$ are holomorphic on the annulus
$\mathcal A_{R_0,\infty}$, then none has an essential singularity at
infinity. Moreover functions $x_j$ extend holomorphically across $\infty$.
\end{proposition}
\begin{proof}
Suppose some coordinate has an essential singularity at infinity.
Fix $\rho>R_0\ge 0$ and let $ M(r):=\max_{1\le j\le m}T_A(\rho,r,x_j)$.
By~\cref{lem:laurent}, $M(r)/\log r\to\infty$. As above, choose a sequence
$r_\nu\to\infty$ such that 
$$
 u_j=\lim_{\nu\to\infty}\frac{T_A(\rho,r_\nu,x_j)}{M(r_\nu)},
 \qquad 0\le u_j\le1,\quad \max_j u_j=1.
$$
Choose $k$ with $u_k=1$. Then $x_k$ has an essential singularity at
infinity. Indeed,  otherwise it is meromorphic there and
$T_A(\rho,r,x_k)=O(\log r)$, contradicting $u_k=1$.

The cyclic identity and the annular estimates from Section~3
give
\begin{equation}\label{eq:annular}
 T_A(\rho,r,f_k\circ x_k)
 \le T_A(\rho,r,x_{k-1})+T_A(\rho,r,x_{k+1})+O(1).
\end{equation}
where again $T_A(\rho,r,x_{0})=T_A(\rho,r,x_{m})$ and $T_A(\rho,r,x_{m+1})=T_A(\rho,r,x_{1})$.
\medskip

If $f_k$ is transcendental, part~(2) of~\cref{lem:fourfiber} gives
$T_A(\rho,r,f_k\circ x_k)/T_A(\rho,r,x_k)\to\infty$.
By replacing $r$ by $r_\nu$ in \eqref{eq:annular} and dividing the equation by $M(r_\nu)$ obtain a contradiction as we let $r_\nu\to\infty$.
\medskip 

If $f_k$ is a polynomial of degree $d_k\ge2$, part~(1)
of~\cref{lem:fourfiber} gives
$$
 d_k u_k\le u_{k-1}+u_{k+1}\le2.
$$
Hence $d_k=2$ and $u_{k-1}=u_{k+1}=1$. In particular, $x_{k+1}$ again has an essential singularity at infinity. By propagating maximal normalized growth around the cycle we eventually
reach an index $j$ for which $f_j$ is a transcendental entire function, and we obtain the same contradiction as before. Therefore each $x_j$ is meromorphic at infinity.
\medskip

Finally let us prove that functions $x_j$ extend holomorphically across $\infty$.
Suppose some coordinate has a pole there, and choose $x_k$ with
maximal pole order $p>0$. If $f_k$ is a polynomial of degree
$d_k\ge2$, then $f_k\circ x_k$ has pole order $d_kp>p$, while the
right-hand side of equation \eqref{eq:cyclic} has pole order at most $p$.
If $f_k$ is transcendental entire, then $f_k\circ x_k$ has an
essential singularity at infinity, while the right-hand side of equation \eqref{eq:cyclic} is
meromorphic there. Both cases are impossible. Thus all singularities
at infinity are removable.

\end{proof}

{\bf Exclusion of $\C^*$:} By applying \cref{prop:exteriorgrowth}  to functions $X_j(t)=x_j(1/t)$ we can see that functions $x_j$ can be holomorphically extended over $0$, hence they are entire. Furthermore, by \cref{prop:planegrowth} to these functions must be constant, which brings us to a contradiction. 
\medskip 

{\bf Exclusion of $\D^*$:} Now assume that $A=\D^*$ and apply \cref{prop:exteriorgrowth} to $X_j(t)=x_j(1/t)$. Thus, we can extend $\gamma:\D^*\to\C^2$ to a non-constant holomorphic function $\hat \gamma:\D\to\C^2$. In order to reach a contradiction it is enough to prove $p:=\hat \gamma(0)\in\Omega$. Indeed, by \cref{thm:abfp} the limit set  $\Sigma=\gamma(\D^*)$ is closed in $\Omega$. The remainder of the proof proceeds exactly as in the polynomial case, see Section~2.

\begin{proof}[The proof of part (1) of \cref{thm:main}] Let us assume that $F$ is conjugate in $\Sc$ to an element in $\Af\cup\Ec$, and that it has an invariant recurrent Fatou component $\Omega$ whose limit set $\Sigma$ has complex dimension one. Up to conjugation by affine map such map can be written in the following form 
$$F(z,w)=(az+f(w),bw), \qquad\text{and}\qquad F(0)=0$$
where either  $|a|< 1$ and $|b|=1$, or else $|a|=1$ and $|b|<1$. Let $f(w)=\sum_{j\geq 1} c_jw^j$ and observe that
\begin{align*}
F^n(z,w)&=\left(a^nz+\sum_{j\geq 1} c_jw^j\left(\sum_{k=0}^{n-1}a^{k}b^{(n-1-k)j}\right), b^nw\right)\\
&=\left(a^nz+ \sum_{j\geq 1} c_j\frac{b^{nj}-a^n}{b^j-a}w^j, b^nw\right).
\end{align*}
Let $h(w)=\sum_{j\geq 1} \frac{c_j}{b^j-a}w^j$ and observe that it is well defined entire function.
If $|a|=1$, then $$F^n(z,w)=a^n(z-h(w),0)+o(1),$$
 and if $|b|=1$, then 
 $$F^n(z,w)=(h(b^nw),b^nw)+o(1),$$ 
 uniformly on compacts in $\C^2$. From here we can easily conclude that $\Omega=\C^2$ and $\Sigma\simeq\C$. 
\end{proof}

\section{Proof of \cref{thm:fix}}
It suffices to prove the theorem for $k=1$, since every iterate $F^k$ is again a finite composition of H\'enon-type maps. Let $F=H_m\circ\cdots\circ H_1$.
If all factors are polynomial, then $\operatorname{Fix}(F)$ is discrete by \cite[Theorem 3.1]{FriedlandMilnor}, and the conclusion follows. We may therefore assume that at least one factor is transcendental.

Suppose $\varphi:S\to\operatorname{Fix}(F)$ is a non-constant proper holomorphic map, and assume that $S$ is not hyperbolic. Clearly $S$ can not be compact, hence by the uniformization theorem there exists a  universal covering $\pi:\C\to S$. Let us define an entire map  $\gamma=\varphi\circ\pi$. Write $\gamma(t)=(x_1(t),x_0(t))$ and define entire functions
$x_2,\ldots,x_{m+1}$ recursively by
\begin{equation}\label{eq:cycle2}
 x_{k+1}(t)=f_k(x_k(t))-\delta_k x_{k-1}(t),
 \qquad 1\leq k\leq m.
\end{equation}
Moreover, since $F\circ\gamma=\gamma$ we obtain
$$x_0(t)=x_m(t)\qquad\text{and} \qquad x_1(t)=x_{m+1}(t).$$
Therefore, the entire functions $x_1,\ldots,x_{m}$ form a cycle and we can prove that they are all constant, by simply following the proof of \cref{prop:planegrowth}, with $\lambda=1$. This contradicts the fact that $\gamma$ is non-constant. Hence $S$ is hyperbolic.

\medskip

We now prove that $S$ has no puncture ends.Suppose otherwise. Then there exists a conformal embedding $\pi:\D^*\to S$ such that $\pi(t)$ leaves every compact subset of $S$ as $t\to 0$. Set $\gamma=\varphi\circ\pi$. As before, using the system of equations \eqref{eq:cycle2} we obtain functions $x_1,\ldots, x_m\in \mathcal{O}(\D^*)$ which form a cycle. By applying the  proof \cref{prop:exteriorgrowth}, to the functions $X_j(t)=x_j(1/t)$ and $\lambda=1$, we can show that $x_j$ extend holomorphically across the origin. Thus, the map $\gamma$ extends to a holomorphic function $\hat\gamma:\D\to\operatorname{Fix}(F)$. Now take any sequence of points $t_n\to 0$ in $\D$. Since $\hat\gamma$ is holomorphic $\hat\gamma(t_n)\rightarrow p\in \operatorname{Fix}(F)$. The set $K=\{\hat\gamma(t_n)\mid n\geq 1\}\cup \{p\}$ is  compact and since $\varphi$ is proper, it follows that $\varphi^{-1}(K)$ is compact in $S$. But by our puncture end assumption $\pi(t_n)$ leaves every compact subset of $S$, which is a contradiction.

\end{document}